\documentclass[sn-mathphys-num]{sn-jnl} 

\usepackage{graphicx}
\usepackage{multirow}
\usepackage{amsmath,amssymb,amsfonts}
\usepackage{amsthm}
\usepackage{mathrsfs}
\usepackage[title]{appendix}
\usepackage{xcolor}
\usepackage{textcomp}
\usepackage{manyfoot}
\usepackage{booktabs}
\usepackage{listings}
\usepackage{hyperref}

\theoremstyle{thmstyleone}
\newtheorem{theorem}{Theorem}[section]

\theoremstyle{thmstyletwo}
\newtheorem{example}[theorem]{Example}
\newtheorem{remark}[theorem]{Remark}
\newtheorem{lemma}[theorem]{Lemma}

\newtheorem{fact}[theorem]{Fact}
\theoremstyle{definition}

\theoremstyle{remark}
\numberwithin{equation}{section}

\theoremstyle{thmstylethree}%

\begin{document}
	
\title[Eigenvalue value bounds]
{Spectral bounds for certain special type of rational matrices}

\author[1]{\fnm{Pallavi} \sur{Basavaraju}}\email{pallavipoorna6@gmail.com; pallavipoorna20@iisertvm.ac.in}
\equalcont{These authors contributed equally to this work.}

\author[2]{\fnm{Shrinath} \sur{Hadimani}}\email{shrinath.hadimani@manipal.edu; srinathsh3320@iisertvm.ac.in}
\equalcont{These authors contributed equally to this work.}
	
\author*[3]{\fnm{Sachindranath} \sur{Jayaraman}}\email{sachindranathj@iisertvm.ac.in; sachindranathj@gmail.com}
\equalcont{These authors contributed equally to this work.}
	
\affil[1]{\orgdiv{Department of Mathematics}, \orgname{Dr. G. Shankar Government Women's First Grade College 
and P.G. Study Centre}, \orgaddress{\street{Ajjarakadu}, \city{Udupi}, \postcode{576101}, \state{Karnataka}, \country{India}}}

\affil[2]{\orgdiv{Department of Mathematics}, \orgname{Manipal Institute of Technology, Manipal Academy of Higher 
Education}, \orgaddress{\city{Manipal}, \postcode{576104}, \state{Karnataka}, \country{India}}}

\affil[3]{\orgdiv{School of Mathematics}, \orgname{Indian Institute of Science Education and Research Thiruvananthapuram}, 
\orgaddress{\street{Vithura}, \city{Thiruvananthapuram}, \postcode{695551}, \state{Kerala}, \country{India}}}

\abstract{The aim of this manuscript is to derive bounds on the moduli of eigenvalues
of special type of rational matrices of the form $T(\lambda) = \displaystyle -B_0 +I\lambda +\frac{B_1}{\lambda-\alpha_1}+ \dots+ \frac{B_m}{\lambda-\alpha_m}$, 
where $B_i$'s are $n \times n$ complex matrices and $\alpha_i$'s are distinct complex 
numbers, using the following methods: $(1)$ an upper bound is obtained using the Bauer-Fike 
theorem for complex matrices on an associated block matrix $C_T$ of the given rational matrix 
$T(\lambda)$, $(2)$ a lower bound is obtained in terms of a zero of a scalar real rational 
function $p(x)$ associated with $T(\lambda)$, using Rouch$\text{\'e}$'s theorem for 
matrix-valued functions and $(3)$ an upper bound is also obtained using a numerical 
radius inequality for a block matrix $C_q$ associated with another scalar real rational
function $q(x)$ corresponding to $T(\lambda)$. These bounds are compared when the 
coefficients are unitary matrices. Numerical examples are given to illustrate the results 
obtained.}

	
	
\maketitle
	
\section{Introduction}\label{sec-1}

An $n \times n$ rational matrix denoted by $T(\lambda)$, is one whose entries are complex 
rational functions. The rational eigenvalue problem is to determine a complex number 
$\lambda_0$ and an $n \times 1$ nonzero vector $v$ such that $T(\lambda_0)v=0$. The 
scalar $\lambda_0$ is called an eigenvalue of $T(\lambda)$. Rational eigenvalue problems, 
abbreviated henceforth as REPs, are an important class of 
nonlinear eigenvalue problems that arise in applications to science and engineering. REPs 
arise for instance, in applications to computing damped vibration modes of an acoustic 
fluid confined in a cavity \cite{Chou-Huang}, describing the eigenvibration of a string 
with a load of mass attached by an elastic string \cite{Betcke-Higham-Tisseur} and in 
application to photonic crystals \cite{Engstrom-Langer-Tretter}, to name a few. Readers 
may refer to \cite{Pallavi-Shrinath-Sachindranath}, \cite{Dopico-Marcaida}, 
\cite{Mehrmann-Heinrich}, \cite{Saad-El-Guide-Miedlar} and \cite{Su-Bai} and the references 
cited therein for some recent work on REPs. Determining the exact eigenvalues of rational 
matrices presents a significant challenge, and hence they are approximated using iterative 
methods \cite{Mehrmann-Heinrich}. Consequently, establishing bounds on eigenvalues is 
crucial for making an initial guess, which influences the convergence rate of the iteration. 
A rational matrix, where each entry is a scalar polynomial, is referred to as a matrix polynomial.
There are several spectral bounds for matrix polynomials depending on norms of 
coefficients and roots of associated scalar polynomials (see for instance, \cite{Cameron}, 
\cite{Higham-Tisseur}, \cite{Monga-Shah} and \cite{Shah-Singh}). For rational 
matrices, one approach to estimate/compute spectral bounds is to convert the rational matrix to a 
matrix polynomial and use existing results. In practice, it is difficult to determine 
the coefficients of the matrix polynomial that comes out from this technique. Among 
the ones available in the literature, there is no easy or the  best way to determine 
the location of eigenvalues of rational matrices using matrix polynomials. The purpose 
of this work is to derive bounds on the moduli of eigenvalues of certain 
special type of rational matrices. We provide bounds that can be calculated with a 
small computational effort.

\medskip
We work either over the field $\mathbb{C}$ of complex numbers or over the field 
$\mathbb{R}$ of real numbers. The vector space of $n \times n$ matrices over 
$\mathbb{C}$ (respectively, $\mathbb{R}$) is denoted by $M_n(\mathbb{C})$ 
(respectively, $M_n(\mathbb{R})$). $||\cdot ||_2$ denotes the spectral norm of a square 
matrix. The condition number of a square matrix $A$ is defined as \\
$\kappa(A) = \begin{cases}
||A||||A^{-1}|| & \text{if $A$ is invertible}\\
\infty & \text{otherwise}
\end{cases}$, where $||\cdot||$ is any matrix norm.

\medskip
The organization of the manuscript is as follows. Section \ref{sec-2.1} contains a brief 
introduction to rational matrices. This is followed by deriving a bound on the moduli 
of eigenvalues of rational matrices of special type using an associated block matrix 
(see Section \ref{sec-2.1} for the definition and  Section \ref{sec-2.2} for details). We 
then derive a lower bound on the moduli of eigenvalues using roots of a real rational 
function (see Section \ref{sec-2.3} for details). Section \ref{sec-2.4} concerns deriving 
a bound using scalar polynomials and the same is done using numerical radius in 
Section \ref{sec-2.5}. These bounds are compared in Section 
\ref{sec-3}. Numerical illustrations are given in Section \ref{sec-4}. The computations are 
done using Matlab.

\medskip
\section{Main results}\label{sec-2}

The main results are presented in this section. We start with preliminaries on rational 
matrices. In the subsections that follow, we derive various bounds on the moduli of 
eigenvalues of rational matrices.

\subsection{Rational matrices}\label{sec-2.1}\hspace*{\fill} 

Any $n \times n$ rational matrix can be expressed in the form 
$\displaystyle T(\lambda) = P(\lambda)-\sum_{i=1}^{k}\frac{s_i(\lambda)}{q_i(\lambda)}E_i$, 
where $P(\lambda)=\displaystyle \sum_{j=0}^{d}A_j\lambda^j$ is an $n \times n$ matrix polynomial 
of degree $d$, $s_i(\lambda)$ and $q_i(\lambda)$ are scalar polynomials of degree $n_i$ and 
$d_i$ respectively, and $A_j$'s, $E_i$'s $\in M_n(\mathbb{C})$. Rational matrices are often 
known as rational matrix functions or matrix rational functions in the literature. An 
$n \times n$ rational matrix $T(\lambda)$ is said to be regular if its determinant does not 
vanish identically. The rational eigenvalue problem (REP) is to find a scalar 
$\lambda_0 \in \mathbb{C}$ and a nonzero vector $v \in \mathbb{C}^n$ such that 
$T(\lambda_0)v =0$, with $T(\lambda)$ being regular and $T(\lambda_0)$ bounded 
(that is, $T(\lambda_0)$ has finite entries). The scalar $\lambda_0$ so obtained is called 
an eigenvalue of $T(\lambda)$ and the vector $v$ is called an eigenvector of $T(\lambda)$ 
corresponding to the eigenvalue $\lambda_0$. Note that if $T(\lambda) = B$, where $B$ is a 
nonsingular matrix, then no complex number is an eigenvalue for $T(\lambda)$. 

\medskip

The nonlinear eigenvalue problem (abbreviated as NEP) seeks to find a scalar 
$\lambda_0$ and a nonzero vector $v \in \mathbb{C}^n$ satisfying $G(\lambda_0)v=0$, where 
$G(\lambda)$ is a regular matrix-valued function ($G(\lambda)$ is square and 
its determinant does not vanish identically), and each entry of $G(\lambda_0)$ is bounded. 
Results on the location of eigenvalues of nonlinear eigenvalue problems 
via the Ger\v{s}gorin-type theorem and the quadratic numerical range techniques can be 
found in \cite{Bindel-Hood} and \cite{Tretter}, respectively. A comprehensive treatment of 
NEPs can be found in \cite{Betcke-Higham-Tisseur}, \cite{Dopico-Marcaida}, \cite{Mehrmann-Heinrich} 
and the references cited therein.
The majority of the nonlinear eigenvalue problems in applications are of the form, 
$G(\lambda)=-B_0+A_0\lambda+A_1f_1(\lambda)+\dots+A_p f_p(\lambda)$,
where $f_i: \Omega \rightarrow \mathbb{C}$ are analytic functions and $\Omega$
is a region in $\mathbb{C}$. In \cite{Saad-El-Guide-Miedlar}, to study the eigenvalue 
problem of $G(\lambda)$ the authors consider the $surrogate$ problem, 
$T(\lambda)v = \Big(-B_0+ A_0 \lambda + \displaystyle \sum_{i=1}^{m} 
\frac{B_i}{\lambda-\alpha_i}\Big)v=0$, where the $\alpha_i$'s are 
distinct complex numbers. This is achieved by approximating each $f_i$ by a rational 
function of the form $r_i(\lambda)= \displaystyle 
\sum_{i=1}^{m} \frac{\sigma_{ij}}{\lambda-\alpha_i}$. The $\alpha_i$'s are the same 
for each $r_i$. This motivates us to study the rational matrices of the form 
\begin{equation}\label{Eqn-2.1}
\displaystyle T(\lambda) = -B_0 +A_0\lambda +\frac{B_1}{\lambda-\alpha_1}+ \dots+ 
\frac{B_m}{\lambda-\alpha_m}, 
\end{equation}
where the $B_i$'s are $n \times n$ matrices and the $\alpha_i$'s are 
distinct complex numbers. Since we are interested in finding bounds on the eigenvalues 
of $T(\lambda)$ we assume $A_0$ to be nonsingular. We assume $A_0=I$, the identity 
matrix as one can multiply Equation \eqref{Eqn-2.1} by $A_0^{-1}$. We, thus, consider 
rational matrices of the form
\begin{equation}\label{eqn-Standard form}
T(\lambda) = \displaystyle -B_0 + I\lambda + \frac{B_1}{\lambda-\alpha_1} +\frac{B_2}
{\lambda-\alpha_2}+\cdots +\frac{B_m}{\lambda-\alpha_m},
\end{equation}
where $\alpha_i$'s are distinct complex numbers ordered $|\alpha_1| < |\alpha_2|<\cdots <|\alpha_m|$ 
and the $B_i$'s are $n\times n$ complex matrices. It turns out that many rational 
eigenvalue problems can be converted to this form (see for instance, Example \ref{Example-3}). 
Note that for $T(\lambda)$ given in Equation \eqref{eqn-Standard form}, the REP can be converted 
to a linear eigenvalue problem, $P(\lambda)v =0$, where $P(\lambda) = I\lambda -C_T$ with
\begin{center}
$C_T = 
\begin{bmatrix}
\alpha_1 I & 0 & \cdots & 0 & -I \\
0 & \alpha_2 I  & \cdots & 0 &-I \\
\vdots & \vdots   & \ddots & \vdots & \vdots \\
0 & 0 & \cdots & \alpha_m I & -I \\
B_1 & B_2 & \cdots & B_m & B_0 & 
\end{bmatrix}$
\end{center}
of size $(m+1)n \times (m+1)n$.  $P(\lambda)$ is also a polynomial
system matrix of $T(\lambda)$ with the state matrix 
$A(\lambda) = \text{diag} \big((\lambda-\alpha_1)I, \ldots, (\lambda-\alpha_m)I\big)$ 
(see \cite{Amparan-Dopico}, \cite{Dopico-Marcaida} for details). Interestingly, 
the linearization of $T(\lambda)$ given in \cite{Alam-Behera} using Fiedler matrices is 
the same as $P(\lambda)$. Thus, corresponding to  $T(\lambda)$ we associate the block matrix 
$C_T$. Note that the eigenvalues of $T(\lambda)$ are also eigenvalues of $C_T$. Moreover, 
if all the $B_i$'s are nonsingular, then $T(\lambda)$ and $C_T$ have the same eigenvalues 
(see \cite{Dopico-Marcaida} for details). It is easy to verify that if one of the 
coefficients $B_i$ is singular, then the corresponding pole $\alpha_i$ is an eigenvalue 
of $C_T$. Therefore, $T(\lambda)$ has at most $(m+1)n$ eigenvalues.

\medskip
In \cite{Chu}, the author proves certain perturbation results for eigenvalues of 
matrix polynomials that are analogous to the Bauer-Fike theorem for complex matrices. 
The author considers a matrix polynomial $P(\lambda)$ and perturbs the coefficient 
matrices to get another matrix polynomial $\widetilde{P}(\lambda)$. Further, the author
defines the spectral variation between the eigenvalues of $P(\lambda)$ and 
$\widetilde{P}(\lambda)$ and derive Bauer-Fike type results on the spectral variation 
using a Jordan triplet of $P(\lambda)$. Similar results were studied in \cite{Chu-Lin} 
for periodic pairs of matrices. Eigenvalue perturbation theory for homogeneous matrix 
polynomials can also be found in \cite{Dedieu-Tisseur}. Note that one can obtain a rational
matrix $\widetilde{T}(\lambda)$ by perturbing the coefficient matrices of the given rational
matrix $T(\lambda)$. However, the block companion matrix corresponding to 
$\widetilde{T}(\lambda)$ and a matrix obtained by perturbing the entries of the block
companion matrix $C_T$ of $T(\lambda)$ are not the same in general.
We do not study perturbation results for rational matrices in this manuscript. 
Instead, we give a region that contains the eigenvalues of a rational matrix 
$T(\lambda)$.

\subsection{Bounds on the eigenvalues of $T(\lambda)$ using Bauer-Fike theorem}\label{sec-2.2}
\hspace*{\fill}

One of the well known results in the perturbation theory of the eigenvalue of a 
diagonalizable matrix is due to Bauer and Fike (Theorem 6.3.2, \cite{Horn-Johnson}). We employ
this result to find a bound for the moduli of eigenvalues of a rational matrix, given in Equation 
\eqref{eqn-Standard form}. We state the Bauer-Fike theorem below.

\medskip
\begin{theorem}[Theorem $6.3.2$, \cite{Horn-Johnson}]\label{Thm-Bauer-Fike}
Let $A \in M_n(\mathbb{{C}})$ be diagonalizable, and suppose that $A=S\Lambda S^{-1}$, in 
which $S$ is nonsingular and $\Lambda$ is diagonal. Let $E \in M_n(\mathbb{C})$ and 
$||\cdot||$ be a matrix norm on $M_n(\mathbb{C})$ that is induced by an absolute norm on 
$\mathbb{C}^n$. If $\hat{\lambda}$ is an eigenvalue of $A+E$, then there is an eigenvalue 
$\lambda$ of $A$ such that $|\hat{\lambda} - \lambda| \leq  \kappa(S)||E||$ in which 
$\kappa(\cdot)$ is the condition number with respect to the matrix norm.
\end{theorem} 

\medskip
The block matrix $C_T$ associated with the rational matrix given in 
Equation \eqref{eqn-Standard form} can be expressed as $C_T = A+E$, where 
$A = \text{diag}(\alpha_1 I, \ldots, \alpha_m I, 0)$ and $E=C_T - A$. Since $A$ is 
diagonalizable, the following result is an easy consequence of the Bauer-Fike theorem. 

\medskip
\begin{theorem}\label{Thm-eigenvalue-Bauer-Fike}
Let $T(\lambda)$ be as in Equation \eqref{eqn-Standard form}, with an eigenvalue $\lambda_0$. 
Then $|\lambda_0| \leq ||E|| + |\alpha_m|$, where 
\begin{center}
$E = \begin{bmatrix}
0 & \cdots & 0 & -I\\
\vdots & \ddots & \vdots & \vdots\\
0 & \cdots & 0 & -I\\
B_1 & \cdots & B_m & B_0\\
\end{bmatrix}$
\end{center}
and $||\cdot||$ is any matrix norm induced by an absolute norm on 
$\mathbb{C}^n$.
\end{theorem}

\medskip
Alternatively, the above result can be obtained by subadditivity of the matrix norm, 
$||C_T|| = ||A+E| |\leq ||A||+||E|| = |\alpha_m| + ||E||$, and the fact that the spectral radius 
of $C_T$ is at most $||C_T||$. In particular, if we assume that the $B_i$'s are unitary 
matrices and the induced norm is the spectral norm, then we get a bound which depends only 
on the number of poles of $T(\lambda)$ and their moduli. We prove this below. The proof is by 
induction on the size of the scalar matrix whose entries are the spectral norms of the block matrices 
from $E$. We shall use the following fact in the proof.

\medskip
\begin{fact}\label{fact-1}
Let $\mathcal{A} = (A_{ij})$ be a block matrix where the entries are complex square matrices with 
$1 \leq i,j \leq m$. Let $\mathcal{\widetilde{A}} = (||A_{ij}||_2)$ (the matrix whose entries are the 
spectral norms of the matrices $A_{ij}$). Then, $||\mathcal{A}||_2 \leq ||\mathcal{\widetilde{A}}||_2$. 
Notice that if $x = (x_1, \dots, x_m) \in \mathbb{C}^n$ 
(partitioned conformally with respect to the size of the matrices) for some $n$, then 
$x$ and $\widetilde{x} = (||x_1||_2, \dots, ||x_m||_2)$ have the same norm. It then easily follows 
that $||\mathcal{A} x||_2 \leq ||\mathcal{\widetilde{A}}  \widetilde{x}||_2$.
\end{fact}

\begin{theorem}\label{Thm-eigenvalue-unitary-Bauer-Fike}
Let $\displaystyle T(\lambda) = -B_0 +I\lambda +\frac{B_1}{\lambda-\alpha_1} + \cdots 
+\frac{B_m}{\lambda-\alpha_m}$, where each $B_i$ is an $n \times n$ unitary matrix 
and the $\alpha_i$'s are distinct complex numbers.  Then $|\lambda_0| \leq 
\displaystyle\Bigg\{\frac{(2m+1)+(4m+1)^{1/2}}{2}\Bigg\}^{1/2} + |\alpha_m|$ for any 
eigenvalue $\lambda_0$ of $T(\lambda)$.
\end{theorem}

\begin{proof}
By Theorem \ref{Thm-eigenvalue-Bauer-Fike}, we have $|\lambda_0| \leq ||E||_2 +  |\alpha_m|$, 
where 
\begin{center}
	$E = \begin{bmatrix}
		0 & \cdots & 0 & -I\\
		\vdots & \ddots & \vdots & \vdots\\
		0 & \cdots & 0 & -I\\
		B_1 & \cdots & B_m & B_0\\
	\end{bmatrix}$.
\end{center} 
Let $\widetilde{E}$ be the matrix 
\begin{center}
	$\widetilde{E} = 
	\begin{bmatrix}
		0 & \cdots & 0 & 1\\
		\vdots & \ddots & \vdots & \vdots\\
		0 & \cdots & 0 & 1\\
		1 & \cdots & 1 & 1\\
	\end{bmatrix}$; 
\end{center} that is, $\widetilde{E}$ is the matrix of size $(m+1) \times (m+1)$,  
whose entries are the spectral norms of the block matrices from $E$. It follows from 
Fact \ref{fact-1} that $||E||_2 \leq ||\widetilde{E}||_2$. It therefore suffices to 
estimate $||\widetilde{E}||_2$ to get a bound on $||E||_2$. Since $\widetilde{E}$ is a 
symmetric matrix, it suffices to compute the eigenvalues of $\widetilde{E}$ to estimate 
$||\widetilde{E}||_2$. To do this, we prove by induction on the size of $\widetilde{E}$ 
that $\text{det} (\widetilde{E} - \lambda I)$ equals 
$(-\lambda)^{m-1} (\lambda^2 - \lambda - m)$.  
When $m = 1, \ \widetilde{E} = \begin{bmatrix}
															0 & 1\\
															1 & 1
                                                           \end{bmatrix}$, 
so that $\text{det}(\widetilde{E} - \lambda I)$ equals 
$(\lambda^2 - \lambda - 1)$. Assume that the result is true for $m = k$; that is, when 
$\widetilde{E} = \begin{bmatrix}
	0 & \cdots & 0 & 1\\
	\vdots & \ddots & \vdots & \vdots\\
	0 & \cdots & 0 & 1\\
	1 & \cdots & 1 & 1\\
\end{bmatrix}$ is a $(k+1) \times (k+1)$ matrix, 
$\text{det} (\widetilde{E} - \lambda I)$ is $(-\lambda)^{k-1} (\lambda^2 - \lambda - k)$. 
Let $m = k+1$ so that 
$\widetilde{E} = \begin{bmatrix}
	0 & \cdots & 0 & 1\\
	\vdots & \ddots & \vdots & \vdots\\
	0 & \cdots & 0 & 1\\
	1 & \cdots & 1 & 1\\
\end{bmatrix}$ is a $(k+2) \times (k+2)$ matrix. Then $\text{det} (\widetilde{E} - \lambda I)$ is 
given by  
\begin{align*}
	\text{det}(\widetilde{E} - \lambda I) = 
	\text{det} \begin{bmatrix}
		-\lambda & \cdots & 0 & 1\\
		\vdots & \ddots & \vdots & \vdots\\
		0 & \cdots & -\lambda & 1\\
		1 & \cdots & 1 & 1 - \lambda\\
	\end{bmatrix}
\end{align*}
\begin{align*}
= (-\lambda) \text{det} \begin{bmatrix}
	-\lambda & \cdots & 0 & 1\\
	\vdots & \ddots & \vdots & \vdots\\
	0 & \cdots & -\lambda & 1\\
	1 & \cdots & 1 & 1 - \lambda\\
\end{bmatrix} \ + \ (-1)^{k+3} \text{det} \begin{bmatrix}
	0 & -\lambda & \cdots & 0\\
	\vdots & \vdots & \ddots & \vdots\\
	0 & 0 & \cdots & -\lambda\\
	1 & 1 & \cdots & 1\\
\end{bmatrix}\\
\end{align*}
$= -\lambda (-\lambda)^{k-1} (\lambda^2 - \lambda - k) + 
(-1)^{2k+5} \ \text{det  \Big(diag}(-\lambda, \cdots, -\lambda)\Big)$,  where 
the first term comes from the induction hypothesis and the second term is by 
expanding the determinant along the first column. On simplification, we get 
$\text{det}(\widetilde{E} - \lambda I) = (-\lambda)^{k}(\lambda^2 - \lambda - (k+1))$. 
Thus, for any positive integer $m$, we have that $\text{det} (\widetilde{E} - \lambda I)$ equals $(-\lambda)^{m-1}(\lambda^2 - \lambda - m)$, whose roots 
are $\lambda = 0, \dfrac{1 \pm (4m+1)^{1/2}}{2}$. We infer from this that 
$||\widetilde{E}||_2 = \dfrac{1 + (4m+1)^{1/2}}{2}$. Note that 
$\Bigg(\dfrac{1 + (4m+1)^{1/2}}{2}\Bigg)^2 = 
\displaystyle\Bigg\{\dfrac{(2m+1)+(4m+1)^{1/2}}{2}\Bigg\}$. We finally conclude  
that $||\widetilde{E}||_2 =  \dfrac{1 + (4m+1)^{1/2}}{2} = 
\displaystyle\Bigg\{\dfrac{(2m+1)+(4m+1)^{1/2}}{2}\Bigg\}^{1/2}$.
\end{proof}	

\medskip

\subsection{Bounds on the eigenvalues of $T(\lambda)$ using rational functions.}
\label{sec-2.3}\hspace*{\fill}

Another method for determining bounds on the eigenvalues is to use 
norms of the coefficient matrices of $T(\lambda)$ and define a rational function whose 
roots are bounds for the eigenvalues of $T(\lambda)$. In \cite{Pallavi-Shrinath-Sachindranath}, 
the authors exploit this technique to derive (only) an upper bound on the set of moduli 
of eigenvalues of a general rational matrix. We state this below (Theorem 
\ref{Thm-eigenvalue-rational function}) for rational matrices of the form given in
Equation \ref{eqn-Standard form}, for the sake of comparison.

\medskip
\begin{theorem}[Theorem $3.8$, \cite{Pallavi-Shrinath-Sachindranath}]\label{Thm-eigenvalue-rational function}
Let $T(\lambda)$ be as given in Equation \eqref{eqn-Standard form}, with an eigenvalue 
$\lambda_0$. Define a real rational function associated with $T(\lambda)$ as follows:
\begin{equation}\label{Eqn-2.3}
q(x) = x -||B_0||-\frac{||B_1||}{x-|\alpha_1|}-\dots-\frac{||B_m||}
{x-|\alpha_m|}, 
\end{equation}
where $||\cdot||$ is any induced matrix norm. Then $|\lambda_0| \leq R$, where $R$ is a 
real positive root of $q(x)$ such that $|\alpha_i| < R$ for all $i=1,2,\cdots,m$.
\end{theorem}

\medskip
In Theorem \ref{Thm-lower bound-eigenvalues}, we derive a lower bound 
on the set of moduli of eigenvalues of a rational matrix as given in Equation 
\ref{eqn-Standard form} that satisfies certain additional assumptions. We do this by 
associating a real rational function, whose positive root gives a lower bound 
for the eigenvalues of rational matrices (satisfying additional assumptions). We make 
use of a Rouch$\text{\'e}$ type theorem for analytic matrix-valued functions 
(Theorem $2.3$ of \cite{Cameron}) and a lemma, whose proof is a simple consequence of 
the intermediate value theorem. We state them in order of preference.

\medskip
\begin{theorem} [Theorem $2.3$, \cite{Cameron}]\label{Thm-Rouche}
Let $A,B: G \rightarrow \mathbb{C}^{n\times n}$, where $G$ is an open connected subset 
of $\mathbb{C}$, be analytic matrix-valued functions. Assume that $A(\lambda)$ is invertible 
on the simple closed curve $\gamma \subseteq G$. If $||A(\lambda)^{-1}B(\lambda)|| < 1$ for 
all $\lambda \in \gamma$, then $A+B$ and $A$ have the same number of eigenvalues inside 
$\gamma$, counting multiplicities. The norm is the matrix norm induced by any norm on 
$\mathbb{C}^n$.
\end{theorem}

\medskip
\begin{lemma}\label{Lem-positive root-rational function}
Let $\displaystyle r(x) = x - a_0 - \frac{a_1}{x-b_1} - \dots-\frac{a_m}{x-b_m}$ be a 
real rational function, where the $a_i$'s are positive and $b_i$'s are nonnegative real 
numbers such that $b_1 < b_2 < \dots < b_m$. Then $r(x)$ has roots $R_1, R_2,\dots, R_{m+1}$ 
such that $R_1 < b_1 < R_2 < b_2 < \dots < R_m < b_m < R_{m+1}$. 
\end{lemma}

We now prove the aforesaid theorem that gives a lower bound for the eigenvalues of 
certain rational matrices.

\medskip
\begin{theorem}\label{Thm-lower bound-eigenvalues}
Let $T(\lambda)$ be as in Equation \eqref{eqn-Standard form}. If $B_0$ is invertible and 
$||B_0^{-1}||^{-1} > 
\displaystyle \frac{||B_1||}{|\alpha_1|} + \dots + \frac{||B_m||}{|\alpha_m|}$, then 
$\tilde{R}\leq |\lambda_0|$ for any eigenvalue $\lambda_0$ of $T(\lambda)$, where 
$\tilde{R}$ is the unique positive root of the real rational function 
$\displaystyle p(x)=x-||B_0^{-1}||^{-1}-\frac{||B_1||}{x-|\alpha_1|}-\dots - 
\frac{||B_m||}{x-|\alpha_m|}$ such that $\tilde{R} < |\alpha_i|$ for all $i= 1, 2 ,\dots,m$.
\end{theorem}

\begin{proof} 
From Lemma \ref{Lem-positive root-rational function}, $p(x)$ has a unique root 
$\tilde{R}$ such that $\tilde{R} < |\alpha_1|$. Note that $\displaystyle p(0) = 
-||B_0^{-1}||^{-1}+\frac{||B_1||}{|\alpha_1|}+ \dots + \frac{||B_m||}{|\alpha_m|} < 0$ 
by the given hypothesis. Therefore, by the intermediate value theorem 
$0 < \tilde{R} < |\alpha_1|$. Let $A(\lambda):= -B_0$ and 
$B(\lambda): = \displaystyle I\lambda + \frac{B_1}{\lambda-\alpha_1} + \dots+
\frac{B_m}{\lambda-\alpha_m}$. Then $T(\lambda) = A(\lambda)+B(\lambda)$. Taking $G$ to 
be the disk $D(0,\tilde{R}):= \{z \in \mathbb{C}: |z| < \tilde{R}\}$, we see that 
$A(\lambda)$ and $B(\lambda)$ are analytic matrix-valued 
functions on $G$. Since $p(0)< 0$ and $p(\tilde{R}) = 0$, we have $p(x) < 0$ for all 
$0 \leq x < \tilde{R}$. Therefore, for all $|\lambda| < \tilde{R}$ we have
\begin{equation}\label{Eqn-rat}
|\lambda| - ||B_0^{-1}||^{-1} - \frac{||B_1||}{|\lambda|-|\alpha_1|} - \dots - 
\frac{||B_m||}{|\lambda|-|\alpha_m|} < 0. 
\end{equation}
Now for $|\lambda| < \tilde{R}$, consider 
\begin{align*}
||B(\lambda)|| & = \Big | \Big | I\lambda +
\frac{B_1}{\lambda-\alpha_1} + \dots+\frac{B_m}{\lambda-\alpha_m} \Big | \Big | \\
& \leq |\lambda| + \frac{||B_1||}{|\lambda-\alpha_1|} +\dots+ \frac{||B_m||}
{|\lambda-\alpha_m|}\\
& \leq |\lambda| + \frac{||B_1||}{|\alpha_1| - |\lambda|} + \dots +
\frac{||B_m||}{|\alpha_m| - |\lambda|} \hspace{1cm} (\text{since} \, |\lambda| < 
|\alpha_i|) \\
& = |\lambda| -\frac{||B_1||}{|\lambda|-|\alpha_1|} -\dots- \frac{||B_m||}{|\lambda|-
|\alpha_m|}\\
& < ||B_0^{-1}||^{-1} = ||A(\lambda)^{-1}||^{-1} \hspace{3cm} \big(\text{by} \, 
\eqref{Eqn-rat}\big).
\end{align*}
For any $\epsilon >0$, define $\gamma:= (\tilde{R}-\epsilon) e^{i\theta}$, where 
$0\leq\theta \leq 2 \pi$. Then $||A^{-1}(\lambda)B(\lambda)|| < 1$ for all 
$\lambda \in \gamma$. Since $\epsilon > 0$ is arbitrary, we see from Theorem 
\ref{Thm-Rouche} that the number of eigenvalues of $A(\lambda)$ and $A(\lambda) + 
B(\lambda)$ are same inside $D(0,\tilde{R})$. However, as there are no eigenvalues of 
$A(\lambda)$  inside $D(0,\tilde{R})$,  $T(\lambda) = A(\lambda) + B(\lambda)$ does not 
have any eigenvalues inside $D(0,\tilde{R})$. Thus, for any eigenvalue $\lambda_0$ of 
$T(\lambda)$ we have $\tilde{R} \leq |\lambda_0|$, thereby giving a lower bound 
as required.
\end{proof}

\subsection{Bounds on the eigenvalues of $T(\lambda)$ using polynomials}\label{sec-2.4}\hspace*{\fill}

Let us now consider yet another well known technique that is used to find bounds on the 
eigenvalues of matrix polynomials. The idea is to convert the rational matrix $T(\lambda)$ 
given in Equation \eqref{eqn-Standard form} into a matrix polynomial by multiplying by 
$\displaystyle \prod_{i=1}^{m}(\lambda-\alpha_i)$. That is, 
\begin{equation}\label{Eqn-2.4}
\displaystyle \prod_{i=1}^{m}(\lambda-\alpha_i) T(\lambda) = P(\lambda), 
\end{equation}
a matrix polynomial. It is easy to verify that the set of eigenvalues of $T(\lambda)$ is 
contained in the set of eigenvalues of $P(\lambda)$. While there are 
many techniques in the literature to determine the eigenvalue location of matrix 
polynomials, we restrict ourselves to only one such method due to 
Higham and Tisseur (Lemma $3.1$, \cite{Higham-Tisseur}).  

\medskip

When $m$ is large, it is difficult to determine the coefficients of the matrix polynomial  
$P(\lambda)$ described in the previous paragraph. We therefore restrict ourselves to the 
case when $m = 1$. This is mainly for the sake of comparison and the proof carries 
over verbatim for arbitrary $m$.

\medskip
\begin{theorem}\label{Thm-eigenvalue-Cauchy-Mason}
Let $\displaystyle T(\lambda)= -B_0 + I\lambda + \frac{B_1}{\lambda-\alpha}$, where $\alpha$ 
is a complex number. Then for any eigenvalue $\lambda_0$ of $T(\lambda)$,
$|\lambda_0| \leq R$, where $R$ is the unique positive root of the polynomial
$u(\lambda)= \lambda^2 -||B_0 +\alpha I||\lambda -||\alpha B_0 +B_1||$.
\end{theorem}

\begin{proof}
Let $P(\lambda) = (\lambda-\alpha)T(\lambda) = I\lambda^2-(B_0 +\alpha I)\lambda +
(\alpha B_0+B_1)$. Note that if $\lambda_0 \in \mathbb{C}$ is an eigenvalue of 
$T(\lambda)$, then $\lambda_0$ is an eigenvalue of $P(\lambda)$. We deduce the desired 
conclusion from Lemma $3.1$, \cite{Higham-Tisseur}.
\end{proof}

\subsection{Estimation of bounds on the largest root of scalar rational function $q(x)$}
\label{sec-2.5}
\hspace{\fill}

Let $T(\lambda)$ be as in Equation \eqref{eqn-Standard form}. By Theorem 
\ref{Thm-eigenvalue-rational function}, we know that the largest real root $R$, 
of the scalar rational function $q(x)$ given in Equation \eqref{Eqn-2.3} is an upper bound 
on the moduli of the eigenvalues of $T(\lambda)$. Since $q(x)$ is a $1 \times 1$ 
rational matrix, its roots are contained in the set of eigenvalues of the 
matrix 
\begin{center}
$C_q=\begin{bmatrix}
|\alpha_1| & 0 & \cdots & 0 & -1 \\
0 & |\alpha_2| & \cdots & 0 & -1 \\
\vdots & \vdots & \ddots & \vdots & \vdots \\
0 & 0 & \cdots & |\alpha_m| & -1 \\
-||B_1|| & -||B_2|| & \cdots & -||B_m|| & ||B_0||
\end{bmatrix}$
\end{center}
of size $(m+1) \times (m+1)$. We now give a bound on $R$ using 
numerical radius inequalities on $C_q$. We begin with a few notations. Given 
$A \in M_n(\mathbb{C})$ the numerical range and the numerical radius of $A$ are denoted 
by $W(A)$ and $w(A)$ respectively and are defined as 
$W(A):= \{ x^*Ax : x \in \mathbb{C}^n \, \text{and}
\, ||x|| =1\}$ and $w(A) = \sup\{|\lambda| : \lambda \in W(A)\}$. 
If $\rho (A)$ denotes the spectral radius of $A$, then for any eigenvalue $\mu_0$ of $A$ 
we have $|\mu_0| \leq \rho(A) \leq w(A)$. 

\medskip

We use the following lemma to estimate a bound on $R$. 

\begin{lemma}[\cite{Abu-Kittaneh}, Lemma $3$]\label{Lem-numerical radius bound-2}
Let $A \in M_k(\mathbb{C}), B \in M_{k,s}(\mathbb{C}), C \in M_{s,k}(\mathbb{C})$ and 
$D \in M_{s}(\mathbb{C})$, and let $K = 
\begin{bmatrix}
A & B \\
C & D
\end{bmatrix}$. Then $w(K) \leq \displaystyle \frac{1}{2}\Big(w(A)+ w(D) + 
\sqrt{(w(A)-w(D))^2 + 4w^2(K_0)}\Big)$, where $K_0 = \begin{bmatrix}
0 & B \\
C & 0
\end{bmatrix}$. 
\end{lemma}

\medskip
We first derive a bound on $R$ when the coefficients are unitary matrices.

\medskip
\begin{theorem}\label{Thm-bound-numerical radius-2}
Let $q(x)$ be a rational function as in Equation \eqref{Eqn-2.3}, and $R$ be the largest 
root of $q(x)$. If $B_1, \ldots , B_m$ are unitary matrices, then 
\begin{center}
$R \leq \displaystyle \frac{1}{2}\Bigg(|\alpha_m|+||B_0||_2 +\sqrt{(|\alpha_m|-||B_0||_2)^2 + 
4 m}\Bigg)$.
\end{center} 
\end{theorem}
	
\begin{proof}
Let $B_1, \ldots , B_m$ be unitary matrices. Then $||B_i||_2=1$ for $1 \leq i \leq m$ and 
\begin{center}
$C_q= \begin{bmatrix}
\begin{array}{c c c c|c}
|\alpha_1| & 0 & \cdots & 0 & -1 \\
0 & |\alpha_2| & \cdots & 0 & -1 \\
\vdots & \vdots & \ddots & \vdots & \vdots \\
0 & 0 & \cdots & |\alpha_m| & -1 \\
\hline
-1 & -1 & \cdots & -1 & ||B_0||_2
\end{array}
\end{bmatrix}  = 
\begin{bmatrix}
A & B \\
C & D
\end{bmatrix}$.
\end{center}
Therefore, $w(A)=|\alpha_m|$ and $w(D)=||B_0||_2$.
Let $K_0 := \begin{bmatrix}
0  & \cdots & 0 & -1 \\
0  & \cdots & 0 & -1 \\
\vdots  & \ddots & \vdots & \vdots \\
0 & \cdots & 0 & -1 \\
-1 & \cdots & -1 & 0
\end{bmatrix}$.
Since $K_0$ is a real symmetric matrix, $\rho(K_0) = w(K_0)$. 
The characteristic polynomial of $K_0$ is 
$\displaystyle t(\lambda)= (-1)^{m+1} \lambda^{m+1}+ 
(-1)^m m \lambda^{m-1}$ so that the eigenvalues of $K_0$ are $0, \, \pm \sqrt{m}$. 
Thus, $\rho(K_0)=w(K_0)= \displaystyle \sqrt{m}$. The desired conclusion then follows 
from Lemma \ref{Lem-numerical radius bound-2}.
\end{proof}

\medskip
In the above theorem, if $B_1, B_2, \ldots, B_m$ are arbitrary matrices, 
then we have the following result.

\medskip
\begin{theorem}\label{Thm-bound-numerical radius-3}
Let $q(x)$ be a rational function as in Equation \eqref{Eqn-2.3}, and $R$ be the 
largest root of $q(x)$. Then 
\begin{center}
$R \leq \displaystyle \frac{1}{2}\Bigg(|\alpha_m|+||B_0|| + 
\sqrt{(|\alpha_m|-||B_0||)^2 + 4 k}\Bigg)$,
\end{center} 
where $k = \max \left\{m , \displaystyle \sum_{i=1}^{m}||B_i||^2\right\}$ and 
$||\cdot ||$ is any induced matrix norm.
\end{theorem}

\begin{proof}
Consider 
\begin{center}	
$C_q =\begin{bmatrix}
\begin{array}{c c c c|c}
		|\alpha_1| & 0 & \cdots & 0 & -1 \\
		0 & |\alpha_2| & \cdots & 0 & -1 \\
		\vdots & \vdots & \ddots & \vdots & \vdots \\
		0 & 0 & \cdots & |\alpha_m| & -1 \\
		\hline
		-||B_1|| & -||B_2|| & \cdots & -||B_m|| & ||B_0||
\end{array}
\end{bmatrix}=
\begin{bmatrix}
A & B \\
C & D
\end{bmatrix} $.
\end{center}
Then $w(A)=|\alpha_m|$ and $w(D)=||B_0||$.
Let $K_0 := \begin{bmatrix}
		0  & \cdots & 0 & -1 \\
		0  & \cdots & 0 & -1 \\
		\vdots  & \ddots & \vdots & \vdots \\
		0 & \cdots & 0 & -1 \\
		-||B_1|| & \cdots & -||B_m|| & 0
	\end{bmatrix}$.
It is easy to verify that $||K_0||_2 = \max \left\{\sqrt{m}, \displaystyle 
\sqrt{\sum_{i=1}^{m}||B_i||^2} \right\}$. Since $ w(K_0) \leq ||K_0||_2$, the 
conclusion follows from Lemma \ref{Lem-numerical radius bound-2}.
\end{proof}

\medskip

\begin{remark}\label{rem-2.5-1}
If $\lambda_0$ is an eigenvalue of $T(\lambda)$ as given in Equation 
\eqref{eqn-Standard form}, then by the above theorem 
$|\lambda_0| \leq \displaystyle \frac{1}{2}\left(|\alpha_m|+||B_0|| + 
\sqrt{(|\alpha_m|-||B_0||)^2 + 4 \displaystyle k}\right)$, where 
$k = \max \left\{m , \displaystyle \sum_{i=1}^{m}||B_i||^2\right\}$.
\end{remark}

\section{Comparison of bounds}\label{sec-3}

In this section, we compare the bounds obtained in Section \ref{sec-2}. In general, we 
cannot determine which method gives a better bound (see Remark \ref{Rem-Bound comparison}) 
for arbitrary matrix coefficients. But when the coefficients of the rational matrices as 
given in Equation \eqref{eqn-Standard form} are unitary matrices and the norm is the 
spectral norm, we have the following comparison of the bounds given in 
Theorems \ref{Thm-eigenvalue-unitary-Bauer-Fike} and 
\ref{Thm-eigenvalue-rational function} and Theorems \ref{Thm-eigenvalue-unitary-Bauer-Fike} 
and \ref{Thm-bound-numerical radius-2}.

\medskip
\begin{theorem}\label{Thm-bound comparison}
Let $\displaystyle T(\lambda)=-B_0+I\lambda+\frac{B_1}{\lambda-\alpha_1} + \dots+\frac{B_m}
{\lambda-\alpha_m}$, where the $B_i$'s are $n \times n$ unitary matrices and 
$\alpha_i$'s are distinct complex numbers. Then the bound given in 
Theorem \ref{Thm-eigenvalue-rational function} 
is better than the bound given in Theorem \ref{Thm-eigenvalue-unitary-Bauer-Fike}.
\end{theorem}

\begin{proof}
Let $R_1$ be the bound for the eigenvalues of $T(\lambda)$ given in Theorem 
\ref{Thm-eigenvalue-unitary-Bauer-Fike}; that is, $R_1 = a + |\alpha_m|$, where 
$a = \displaystyle \Bigg\{\frac{(2m+1)+(4m+1)^{1/2}}{2}\Bigg\}^{1/2}$. Let $R_2$ be the 
eigenvalue bound given in Theorem \ref{Thm-eigenvalue-rational function}, which is the 
unique root of the real rational function 
$\displaystyle q(x)=x- 1 -\frac{1}{x-|\alpha_1|}-\dots - \frac{1}{x-|\alpha_m|}$ such 
that $|\alpha_i| < R_2$ for all $i=1,2,\dots,m$. If $m=1$ and $\alpha_1 = 0$, then 
$R_1 = R_2 = \displaystyle \left(\frac{3+\sqrt{5}}{2}\right)^{1/2}$. 
Therefore, we assume that $m \geq 1$ and $\alpha_m \neq 0$. 
Note that $R_1 > |\alpha_m|$ and $\displaystyle \lim_{x \rightarrow |\alpha_m|^+}q(x) = - \infty$. 
Consider
\begin{align*}
q(R_1) 
& = q(a+|\alpha_m|) \\
& = a+|\alpha_m|-1 -\frac{1}{a + |\alpha_m| -|\alpha_1|}-\dots - \frac{1}{a +|\alpha_m|-
|\alpha_m|}\\
& = |\alpha_m|+ \frac{a^2-a-1}{a} -\frac{1}{a + |\alpha_m| -|\alpha_1|}-\dots - \frac{1}
{a +|\alpha_m|-
|\alpha_{m-1}|} \\
& > |\alpha_m| + \frac{a^2-a-1}{a} -\frac{1}{a}- \dots - \frac{1}{a} \\
& = |\alpha_m| + \frac{a^2-a-1}{a} -\frac{m-1}{a} = |\alpha_m| + \frac{a^2-a-m}{a}.
\end{align*}
Since $m \geq 1$, the only positive root of $x^2 -x -m =0$ is 
$x_0 = \displaystyle \frac{1+ (4m +1)^{1/2}}{2}$. On squaring we get, 
$x_0^2 = \displaystyle \frac{(2m+1)+ (4m +1)^{1/2}}{2} = a^2$. Therefore, 
$x_0 = a = \displaystyle \Bigg\{\frac{(2m+1)+ (4m +1)^{1/2}}{2}\Bigg\}^{1/2}$ is a root 
of $x^2 -x -m = 0$ and hence $a^2-a-m = 0$. This in turn implies $q(R_1) > |\alpha_m| > 0$. 
The intermediate value theorem ensures that $q(x)$ has a root in 
$(|\alpha_m|,R_1)$. But $R_2$ is the only root of $q(x)$ such that $|\alpha_m| < R_2$. 
Therefore, $R_2 \in (|\alpha_m|,R_1)$. Hence, $R_2 < R_1$.
\end{proof}

\medskip
\begin{theorem}\label{Thm-bound comparison-2}
Let $\displaystyle T(\lambda)=-B_0+I\lambda+\frac{B_1}{\lambda-\alpha_1} 
+ \dots+\frac{B_m}{\lambda-\alpha_m}$, where the $B_i$'s are $n \times n$ 
unitary matrices and the $\alpha_i$'s are distinct complex numbers. Then the 
bound given in Theorem \ref{Thm-bound-numerical radius-2} is better than the bound 
given in Theorem \ref{Thm-eigenvalue-unitary-Bauer-Fike}, that is,

\begin{center}
	$\displaystyle \frac{1}{2}\Big(1+|\alpha_m|+\sqrt{(|\alpha_m|-1)^2+4m}\Big) \leq \Bigg\{\frac{(2m+1)+(4m+1)^{1/2}}{2}\Bigg\}^{1/2} + |\alpha_m|$.
\end{center}
\end{theorem}
	
\begin{proof}
Let $R_1= a + |\alpha_m|$ be the bound for the eigenvalues of 
$T(\lambda)$ given in Theorem \ref{Thm-eigenvalue-unitary-Bauer-Fike}, 
where $a = \displaystyle \Bigg\{\frac{(2m+1)+(4m+1)^{1/2}}{2}\Bigg\}^{1/2}$. 
Let $R_4$ be the bound on the eigenvalues of $T(\lambda)$ obtained in 
Theorem \ref{Thm-bound-numerical radius-2}. Since the $B_i$'s are unitary matrices, 
$R_4 = \displaystyle \frac{1}{2}\Big(1+|\alpha_m|+\sqrt{(|\alpha_m|-1)^2+4m}\Big)$. If $m=1$ 
and $\alpha_1=0$, then $R_1 = R_4 = \displaystyle \left(\frac{3 + \sqrt{5}}{2}\right)^{1/2}$. Consider 
$m \geq 1$ and $\alpha_m \neq 0$. Define a scalar rational function 
$w(x):=x -1 -\displaystyle \frac{m}{x-|\alpha_m|}$. The only zeros of 
$w(x)$ are $R_4' =\displaystyle \frac{1}{2}\Big(1+|\alpha_m|-
\sqrt{(|\alpha_m|-1)^2+4m}\Big)$ and $R_4=\displaystyle \frac{1}{2}
\Big(1+|\alpha_m|+\sqrt{(|\alpha_m|-1)^2+4m}\Big)$. By Lemma \ref{Lem-positive root-rational function}, 
$R_4' \in (-\infty, |\alpha_m|)$ and $R_4 \in (|\alpha_m|, \infty)$. The remaining 
part of the proof follows as in the previous theorem.
\end{proof}

\medskip
Some remarks are in order. In what follows, we work with the spectral norm.

\medskip
\begin{remark}\label{Rem-Bound comparison}\hfill

\begin{enumerate}
\item In Theorems \ref{Thm-bound comparison} and \ref{Thm-bound comparison-2}, 
if the matrices are not unitary, we cannot determine which theorem gives 
the better bound. For example, consider $T(\lambda) = \displaystyle 
-B_0+I\lambda+\frac{B_1}{\lambda-\alpha_1}$, where $B_0= 
\begin{bmatrix}
0 & 0\\ 
1 & 0
\end{bmatrix}$ and $B_1= 
\begin{bmatrix}
1 & 0\\
0 & 0
\end{bmatrix}$. 
\begin{itemize}
\item[a.] If $\alpha_1=0.1$, then the bound given in Theorem 
\ref{Thm-eigenvalue-Bauer-Fike} is $R_1 = 1.51$ and the bound given in 
Theorem \ref{Thm-eigenvalue-rational function} is $R_2 =1.65$. Therefore, 
$R_1 < R_2$.
\item[b.]If $\alpha_1=1$, then the bound given in Theorem 
\ref{Thm-eigenvalue-Bauer-Fike} is $R_1 = 2.41$ and the bound given in 
Theorem \ref{Thm-eigenvalue-rational function} is $R_2 =2$. Therefore, 
$R_2 < R_1$.
\item[c.] Again, if $\alpha_1=0.1$, the bound given in Theorem \ref{Thm-bound-numerical radius-2} 
is $R_4 = 1.65$. Therefore, $R_1 < R_4$. If $\alpha = 1$ the bound given in 
Theorem \ref{Thm-bound-numerical radius-2} is $R_4 = 2$. Hence, $R_4 < R_1$.
\end{itemize} 

\item We cannot say which theorem gives a better bound between 
Theorem \ref{Thm-eigenvalue-unitary-Bauer-Fike} and Theorem 
\ref{Thm-eigenvalue-Cauchy-Mason}, even when the coefficients of 
$T(\lambda)$ are unitary matrices. For example consider, $T(\lambda) =
\displaystyle -B_0+I\lambda+\frac{B_1}{\lambda-\alpha_1}$, where 
$B_0 = B_1 = I_2$.
\begin{itemize}
\item[a.] If $\alpha_1 = 1$, then the bound given in Theorem 
\ref{Thm-eigenvalue-unitary-Bauer-Fike} is $R_1 = 2.62 $ and the bound 
given in Theorem \ref{Thm-eigenvalue-Cauchy-Mason} is $R_3 =1+ \sqrt{3} 
= 2.73$. Therefore, we have $R_1 < R_3$.
\item[b.] If and $\alpha_1 = i$, then the bound given in Theorem 
\ref{Thm-eigenvalue-unitary-Bauer-Fike} is $R_1 = 2.62 $ and the bound 
given in Theorem \ref{Thm-eigenvalue-Cauchy-Mason} is $ R_3 = 2.09$. 
In this case, $R_3 < R_1$. 
\end{itemize}

\item The same phenomenon happens with Theorems 
\ref{Thm-eigenvalue-rational function} and \ref{Thm-eigenvalue-Cauchy-Mason}. 
Consider the same example as in $(2)$.
\begin{itemize}
\item[a.] If $\alpha_1 = -1.5$, then the bounds given in Theorems 
\ref{Thm-eigenvalue-rational function} and \ref{Thm-eigenvalue-Cauchy-Mason} 
are $R_2 = 2.28$ and $R_3  = 1$ respectively. Thus, $R_3 < R_2$.
\item[b.] If $\alpha_1 = 1.5$, then the bounds given in Theorems 
\ref{Thm-eigenvalue-rational function} and \ref{Thm-eigenvalue-Cauchy-Mason} 
are $R_2 = 2.28$ and $R_3 =3.27 $ respectively. In this case, we have $R_2 < R_3$. 
\end{itemize}

\item The bounds obtained in Theorem \ref{Thm-bound-numerical radius-2} and Theorem 
\ref{Thm-eigenvalue-Cauchy-Mason} are also not comparable. Consider the same 
example given in $(2)$.
\begin{itemize}
\item[a.] If $\alpha_1 = 1$, then bounds obtained in Theorems \ref{Thm-bound-numerical 
radius-2} and \ref{Thm-eigenvalue-Cauchy-Mason} are $R_4 = 2$ and $R_3 = 2.73$ 
respectively. Therefore, $R_4 < R_3$.
\item[b.] If $\alpha_1 = -0.5$, then bounds obtained in Theorems \ref{Thm-bound-numerical 
radius-2} and \ref{Thm-eigenvalue-Cauchy-Mason} are $R_4 = 1.78$ and $R_3 = 1$ respectively. 
Therefore, $R_3 < R_4$.
\end{itemize}   

\end{enumerate}
\end{remark}

\medskip 
\section{Numerical results}\label{sec-4}\hspace*{\fill}
 
Bounds on the moduli of eigenvalues of rational matrices are less studied 
than bounds on the moduli of eigenvalues of matrix polynomials in the literature. 
Recently in \cite{Pallavi-Shrinath-Sachindranath}, the authors discuss some interesting 
techniques to derive eigenvalue bounds for general rational matrices. There are methods to 
determine approximate eigenvalues in specific regions using iterative methods and rational 
approximation methods (see for instance, \cite{Saad-El-Guide-Miedlar}, \cite{Su-Bai}), 
which are entirely different problems from ours. However, one can convert the rational 
matrix to a matrix polynomial and use existing results in the literature on matrix 
polynomials to compare these bounds. We do this and compare our bounds with the bounds 
given in \cite{Le-Du-Nguyen}. We present three examples of REPs and compare the bounds 
obtained in Section \ref{sec-2} with the bounds given in \cite{Pallavi-Shrinath-Sachindranath} 
and \cite{Le-Du-Nguyen}. In Example \ref{Example-1}, we see that some of the bounds 
obtained in this manuscript are better than the bounds given in 
\cite{Pallavi-Shrinath-Sachindranath} and \cite{Le-Du-Nguyen}. However, in general, 
any one of the above methods is not consistently better than the others.

\medskip
\begin{example}\label{Example-1}
Let $T(\lambda)= -B_0 + I \lambda + \displaystyle \frac{B_1}{\lambda -0.1}$, where $I$ is the 
identity matrix of size $3$, $B_0 = \begin{bmatrix}
	2 & -1 & 0 \\
	-1 & 2 & -1 \\
	0 & -1 & 1
\end{bmatrix}$ and $B_1 = \begin{bmatrix}
-1 & 0 & 1\\
0 & -1 & 1 \\
-1 & 0 & 1
\end{bmatrix}$. Note that the maximum of the moduli of eigenvalues of $T(\lambda)$ is $3.54$.

\medskip

\begin{table}[ht]
\centering
\begin{tabular}{l | l | @{\hspace{0.4cm}} | l | l}
\hline
{Results} & {Bounds} & {Results} & {Bounds} \\
\hline
\hline
Theorem \ref{Thm-eigenvalue-Bauer-Fike} & 3.70 & Theorem $3.2$ of \cite{Le-Du-Nguyen} 
& 4.15\\
Theorem \ref{Thm-eigenvalue-rational function} (Theorem $3.8$, \cite{Pallavi-Shrinath-Sachindranath}) & 3.83 & Corollary $3.2.1$ of \cite{Le-Du-Nguyen} 
& 5.32 \\
Theorem \ref{Thm-eigenvalue-Cauchy-Mason} & 3.90 & Theorem $3.3$ of \cite{Le-Du-Nguyen} & 4.35 \\
Theorem \ref{Thm-bound-numerical radius-3} & 4.36 & Theorem $3.4$ of \cite{Le-Du-Nguyen} & 4.12 \\
Theorem $3.9 (1)$ of \cite{Pallavi-Shrinath-Sachindranath} & 5.42 & Corollary $3.4.2$ 
of \cite{Le-Du-Nguyen} & 5.07 \\
Theorem $3.9 (2)$ of \cite{Pallavi-Shrinath-Sachindranath} & 4.25 & Corollary $3.4.4$ 
of \cite{Le-Du-Nguyen} & 3.99 \\
Theorem $3.9 (3)$ of \cite{Pallavi-Shrinath-Sachindranath} & 3.91 & Corollary $3.4.6$ 
of \cite{Le-Du-Nguyen} & 4.91 \\
Corollary $3.11$ of \cite{Pallavi-Shrinath-Sachindranath} & 5.37 & Theorem $3.6$ 
of \cite{Le-Du-Nguyen} & 4.35 \\
\hline
\end{tabular}
\vspace{0.2cm}
\caption{Bounds obtained from Section \ref{sec-2} and references \cite{Pallavi-Shrinath-Sachindranath}, 
\cite{Le-Du-Nguyen} for Example \ref{Example-1}.}
\label{table:1}
\end{table}

From Table \ref{table:1}, we can conclude that the bound obtained using Theorem 
\ref{Thm-eigenvalue-Bauer-Fike} in this manuscript is better than other bounds for Example \ref{Example-1}. 
\end{example}

\newpage
\medskip
\begin{example}\label{Example-2}
Let $T(\lambda)= -B_0 + I \lambda + \displaystyle \frac{B_1}{\lambda -2}$, where $I$ is the 
identity matrix of size $3$, $B_0 = \begin{bmatrix}
	1 & 0 & 0 \\
	0 & 1 & 0 \\
	0 & 0 & 2
\end{bmatrix}$ and $B_1 = \begin{bmatrix}
	0 & 0 & 0 \\
	0 & 0 & 0 \\
	0 & 0 & -1
\end{bmatrix}$. The maximum of the moduli of eigenvalues of $T(\lambda)$ is $3.00$.

\medskip
\begin{table}[h!]
\centering
\begin{tabular}{l | l | @{\hspace{0.4cm}} | l | l}
\hline
{Results} & {Bounds} & {Results} & {Bounds} \\
\hline
\hline
Theorem \ref{Thm-eigenvalue-Bauer-Fike} & 4.41 & Theorem $3.2$ of \cite{Le-Du-Nguyen} & 4.83 \\
Theorem \ref{Thm-eigenvalue-rational function} (Theorem $3.8$, \cite{Pallavi-Shrinath-Sachindranath}) & 3.00 & Corollary $3.2.1$ of \cite{Le-Du-Nguyen} & 7.99\\
Theorem \ref{Thm-eigenvalue-Cauchy-Mason} & 4.65 & Theorem $3.3$ of \cite{Le-Du-Nguyen} & 5.00  \\
Theorem \ref{Thm-bound-numerical radius-3} & 3.00 & Theorem $3.4$ of \cite{Le-Du-Nguyen} & 4.79 \\
Theorem $3.9 (1)$ of \cite{Pallavi-Shrinath-Sachindranath} & 3.00 & Corollary $3.4.2$ of \cite{Le-Du-Nguyen} & 6.32 \\
Theorem $3.9 (2)$ of \cite{Pallavi-Shrinath-Sachindranath} & 3.00 & Corollary $3.4.4$ of \cite{Le-Du-Nguyen} & 4.14 \\
Theorem $3.9 (3)$ of \cite{Pallavi-Shrinath-Sachindranath} & 3.00 & Corollary $3.4.6$ of \cite{Le-Du-Nguyen} & 5.12 \\
Corollary $3.11$ of \cite{Pallavi-Shrinath-Sachindranath} & 3.50 & Theorem $3.6$ of \cite{Le-Du-Nguyen} & 5.00 \\
\hline
\end{tabular}
\vspace{0.2cm}
\caption{Bounds obtained from Section \ref{sec-2} and references \cite{Pallavi-Shrinath-Sachindranath}, 
\cite{Le-Du-Nguyen} for Example \ref{Example-2}.}
\label{table:2}
\end{table}

In Example \ref{Example-2}, the bound obtained from Theorem 
\ref{Thm-bound-numerical radius-3} of this manuscript is better than bounds 
obtained using methods given in \cite{Le-Du-Nguyen} and it actually 
coincides with the bounds obtained from Theorems $3.8$ and $3.9$ of \cite{Pallavi-Shrinath-Sachindranath}.
\end{example}

\medskip
\begin{example}\label{Example-3}
The following example arises in the finite element 
discretization of a boundary problem describing the eigenvibration of a string with a 
load of mass attached by an elastic spring. We refer readers to \cite{Betcke-Higham-Tisseur} 
for details about this particular REP. Let 
$T(\lambda)v=\left(A-B\lambda+C\frac{\lambda}{\lambda-\alpha}\right)v,$ 
where $A= 
\begin{bmatrix}
	6 & -3 & 0\\
	-3 & 6 & -3\\
	 0 & -3 & 3
\end{bmatrix}$, $B=\frac{1}{18}\begin{bmatrix}
	4 & 1 & 0 \\
	1 & 4  &  1\\
	0 & 1 & 2
\end{bmatrix}$, $C =\begin{bmatrix}
	0 & 0 & 0\\
	0 & 0 & 0 \\
	0 & 0 & 1
\end{bmatrix}$ and $\alpha =1$. Note that $B$ is invertible, therefore the above 
REP is the same as 
$\left(-(A+C)B^{-1}+I\lambda-CB^{-1}\frac{\alpha}{\lambda-\alpha}\right)v=0$. The maximum of 
the moduli of eigenvalues of $T(\lambda)$ is $94.60$.
	
\medskip
	
\begin{table}[h!]
\centering
\begin{tabular}{l | l | @{\hspace{0.4cm}} | l | l}
\hline
{Results} & {Bounds} & {Results} & {Bounds} \\
\hline
\hline
Theorem \ref{Thm-eigenvalue-Bauer-Fike} & 98.46 & Theorem $3.2$ of \cite{Le-Du-Nguyen} & 99.24 \\
Theorem \ref{Thm-eigenvalue-rational function} (Theorem $3.8$, \cite{Pallavi-Shrinath-Sachindranath}) & 97.38 & Corollary $3.2.1$ of \cite{Le-Du-Nguyen} & 191.47 \\
Theorem \ref{Thm-eigenvalue-Cauchy-Mason} & 99.18 & Theorem $3.3$ of \cite{Le-Du-Nguyen} & 99.25 \\
Theorem \ref{Thm-bound-numerical radius-3} & 98.46 & Theorem $3.4$ of \cite{Le-Du-Nguyen} & 98.47 \\
Theorem $3.9 (1)$ of \cite{Pallavi-Shrinath-Sachindranath} & 108.04 & Corollary $3.4.2$ of \cite{Le-Du-Nguyen} & 101.13  \\
Theorem $3.9 (2)$ of \cite{Pallavi-Shrinath-Sachindranath} & 98.27 & Corollary $3.4.4$ of \cite{Le-Du-Nguyen} & 97.33 \\
Theorem $3.9 (3)$ of \cite{Pallavi-Shrinath-Sachindranath} & 97.63 & Corollary $3.4.6$ of \cite{Le-Du-Nguyen} & 98.33 \\
Corollary $3.11$ of \cite{Pallavi-Shrinath-Sachindranath} & 146.40  & Theorem $3.6$ of \cite{Le-Du-Nguyen} & 99.25\\
\hline
\end{tabular}
\vspace{0.2cm}
\caption{Bounds obtained from Section \ref{sec-2} and references \cite{Pallavi-Shrinath-Sachindranath}, 
\cite{Le-Du-Nguyen} for Example \ref{Example-3}.}
\label{table:3}
\end{table}	

Table \ref{table:3} shows that the bound obtained from Corollary $3.4.4$ of \cite{Le-Du-Nguyen} 
is sharper than other bounds for Example \ref{Example-3}. 
	
\end{example}

\medskip
In all three examples the upper bound obtained using Theorem $2.1 (4)$ of \cite{Bini} coincides 
with the bound given in Theorem \ref{Thm-eigenvalue-Cauchy-Mason}. Let us 
point out that Roy and Bora \cite{Roy-Bora} also study the eigenvalue location of quadratic 
matrix polynomials and compare their bounds with that of \cite{Bini}; however, the bounds 
in \cite{Bini} are better than that of \cite{Roy-Bora}. In order to find a lower bound using 
Theorem \ref{Thm-lower bound-eigenvalues}, the coefficient matrices should satisfy the 
hypothesis given in the theorem. Note that the coefficient matrices 
of rational matrices given in Examples \ref{Example-1} and \ref{Example-3} do not 
satisfy the hypothesis of Theorem \ref{Thm-lower bound-eigenvalues}. Therefore, we find 
a lower bound only for Example \ref{Example-2}. Note that the minimum of the moduli of 
eigenvalues of $T(\lambda)$ given in Example \ref{Example-2} is $1$. A lower bound on 
the moduli of eigenvalues obtained from Theorem \ref{Thm-lower bound-eigenvalues} 
is $0.38$.
	
\bmhead{Supplementary information}
The computations were carried out in MATLAB. The MATLAB code files along with the PDF output are available at a GitHub repository.

\section*{Declarations}

\noindent
\textbf{Code availability} The MATLAB code files along with the PDF output is available 
at \url{https://github.com/sachindranathj/MATLAB-Code-J.-Analysis.git}.

\bibliographystyle{sn-mathphys.bst}

\end{document}